\documentclass[a4paper,14pt, reqno]{amsart} 
\usepackage{bbm}

\numberwithin{equation}{section}
\usepackage{mathrsfs}
\usepackage{amsmath,amssymb, amsthm}
\usepackage[all,poly,knot]{xy}
\usepackage[colorlinks,linkcolor=black,anchorcolor=black,citecolor=black]{hyperref}
\usepackage{color}
\newtheorem{thm}{Theorem}[section]

\newtheorem{lem}[thm]{{Lemma}}

\newtheorem{rem}[thm]{{Remark}}

\def\C{\mathscr C}

\def\Ac{\mathrm{Aut}_{\C}}

\def\Hom{\mathrm{Hom}}
\def\Homc{\Hom_{\C}}

\def\Id{\mathrm{Id}}

\def\Mor{\mathrm{Mor}}
\def\Obj{\mathrm{Obj}}

\def\lbr{\left(\begin{array}{c}}
\def\lbrt{\left(\begin{array}{cc}}
\def\lbrth{\left(\begin{array}{ccc}}
\def\rbr{\end{array}\right)}

\title[The spectrum of the singularity category]{The spectrum of the singularity category of a category algebra}
\author{Ren Wang}
\address{School of Mathematical Sciences, University of Science and Technology of China, Hefei, Anhui 230026, P. R. China}
\email{renw@mail.ustc.edu.cn}
\date{\today}

\begin{document}
\begin{abstract}
Let $\C$ be a finite projective EI category and $k$ be a field. The singularity category of the category algebra $k\C$ is a tensor triangulated category. We compute its spectrum in the sense of Balmer.
\end{abstract}
%n--------------------------------
%n--------------------------------

\maketitle
%n--------------------------------
%n--------------------------------

\section{Introduction}
Let $k$ be a field, and $\C$ be a finite skeletal EI category; see \cite{PWebb2}. Here, finite means that $\C$ has
only finitely many morphisms, and the EI condition means
that all endomorphisms in $\C$ are isomorphisms. In particular,
$\Homc(x,x)=\Ac(x)$ is a finite group for each object $x$. Denote by $k\Ac(x)$ the group algebra. 
 
 Denote by $k\C$-mod the category of finitely generated left $k\C$-modules. Denote by $k\C$-proj (resp. $k\C$-Gproj) the full subcategory of $k\C$-mod consisting of all projective (resp. Gorenstein-projective) modules, and denote by $k\C\text{-{\underline{\rm Gproj}}}$ the corresponding stable category modulo projectives. Denote by  ${\rm D}^b(k\C)={\rm D}^b(k\C\text{-mod})$ the bounded derived category of $k\C$-mod. Recall from \cite{Buch} that the singularity category of $k\C$ is the Verdier quotient category ${\rm D}_{\rm sg}(k\C)={\rm D}^b(k\C)/{\rm D}^b(k\C\text{-{\rm proj}})$.
 
 In recent decades, the theory of tensor triangulated geometry has been studied and developed; see \cite{Ba,Kl,St} for instance. It has  important applications in algebraic geometry, algebraic topology and representation theory.

Recall that ${\rm D}^b(k\C)$ is a tensor triangulated category; see~\cite{XF2}. 
Denote by ${\rm Spc} {\rm D}^b(k\C)$ the set of all prime ideals of ${\rm D}^b(k\C)$, which can be topologized; see~\cite{Ba,XF2}. The obtained topological space ${\rm Spc} {\rm D}^b(k\C)$ is called the \emph{spectrum} of ${\rm D}^b(k\C)$.  
Recall from ~\cite[Theorem 3.3.1]{XF2} that the spectrum ${\rm Spc} {\rm D}^b(k\C)$ of ${\rm D}^b(k\C)$ is homeomorphic to $\bigsqcup\limits_{x\in \C} {\rm Spc} {\rm D}^b(k\C_x)$, the disjoint union of the spectrum of ${\rm D}^b(k\C_x)$, where $\C_x$ is the full subcategory of $\C$ with object $\{x\}$. We give a different proof of this result via Verdier quotient functors (or called localization functors); see Theorem~\ref{S}.

Recall from ~\cite{WR} that $\C$ is \emph{projective over $k$} if each
$k{\rm Aut}_{\C}(y)$-$k{\rm Aut}_{\C}(x)$-bimodule $k{\rm Hom}_{\C}(x,y)$ is projective on both sides. %If $\C$ is projective, then we have a tensor triangle equivalence $k\C\text{-{\underline{\rm Gproj}}} \overset{\sim}{\longrightarrow} {\rm D}_{sg}(k\C)$; see~\cite{WR1,WR2}.

Let $\C$ be a finite transporter category. Recall from ~\cite[Theorem 4.2.1]{XF2} that there is a homeomorphism between
${\rm Spc} (k\C\text{-{\underline{\rm Gproj}}})$ and $ \bigsqcup\limits_{x\in \C} {\rm Spc} (kG_x\text{-{\underline{\rm mod}}})
$, which is
a disjoint union, and where $G_x=\Ac(x)$, and $kG_x\text{-{\underline{\rm mod}}}$ is the stable category modulo projectives. Recall from ~\cite{WR} that a finite transporter category is a finite projective EI category. We generalize the above result to finite projective EI categories; see Theorem~\ref{PS}.

%\begin{thm}\label{PS1}
	%Let $\C$ be a finite projective EI category. Then there is a homeomorphism
	%\[{\rm Spc} {\rm D}_{\rm sg}(k\C)\overset{\sim}{\longrightarrow} \bigsqcup\limits_{x\in \C} {\rm Spc} (kG_x\text{-{\underline{\rm mod}}}),\] where the right hand side is a disjoint union, and $G_x=\Ac(x)$. 
%\end{thm}

\section{Tensor triangular geometry}

Recall from \cite{Ba, XF2} that a \emph{tensor triangulated category} is a triple ($\mathscr K$, $\otimes$, $\text{1}$) consisting of a triangulated category $\mathscr K$, a symmetric monoidal (tensor) product $\otimes : \mathscr K\times \mathscr K\rightarrow \mathscr K$, which is exact in each variable and with respect to which there exists an identity $\text{1}$.                     

A \emph{tensor triangulated functor} $F:\mathscr K\rightarrow \mathscr K'$ is an exact functor respecting the monoidal structures and preserves the tensor identity.

Let $\mathscr K$ be a tensor triangulated category. A subcategory $\mathcal I$ of $\mathscr K$ is a \emph{tensor ideal} if it is a thick triangulated subcategory which is closed under tensoring with objects in $\mathscr K$. A tensor ideal $\mathcal P$ of $\mathscr K$ is said to be \emph{prime} if $\mathcal P$ is properly contained in $\mathscr K$ and $x\otimes y\in \mathcal P$ implies either $x\in \mathcal P$ or $y\in \mathcal P$.

Denote by ${\rm Spc} \mathscr K$ the set of all prime ideals of $\mathscr K$. If $x\in \mathscr K$, its \emph{support} is defined to be
\[{\rm supp}_{\mathscr K}(x)=\{\mathcal P\in {\rm Spc} \mathscr K\mid x\notin \mathcal P\}.\]
One can topologize ${\rm Spc} \mathscr K$ by asking the following to be an open basis
\[U(x)={\rm Spc} \mathscr K-{\rm supp}_{\mathscr K}(x)=\{\mathcal P\in {\rm Spc} \mathscr K\mid x\in \mathcal P\}.\]
Indeed, every quasi-compact open subset of ${\rm Spc} \mathscr K$ is of the form $U(x)$ for some $x\in \mathscr K$; see~\cite{Ba, XF2}.

Let $q : \mathscr K\rightarrow {\mathscr K}/ {\mathcal I}$ be a localization functor, where $\mathscr K$ is a tensor triangulated category, $\mathcal I$ is a tensor ideal of $\mathscr K$ and ${\mathscr K}/ {\mathcal I}$ is the corresponding Verdier quotient category. The category ${\mathscr K}/ {\mathcal I}$ inherits the tensor structure of $\mathscr K$; see~\cite[Remark 3.10]{Ba}.

The following lemma is well-known; see~\cite[Propositions 3.6 and 3.11]{Ba}.
\begin{lem}\label{VH}
	Let $q : \mathscr K\rightarrow {\mathscr K} / {\mathcal I}$ be a localization functor. Then we have the following statements.
\begin{enumerate}
\item The map ${\rm Spc}(q) : {\rm Spc}({\mathscr K}/ {\mathcal I})\longrightarrow {\rm Spc}(\mathscr K)$ sending $Q$ to $q^{-1}(Q)$, the original image of $Q$ in the map $q$, induces a homeomorphism between ${\rm Spc}({\mathscr K} / {\mathcal I})$ and the subspace $\{\mathcal P\in {\rm Spc} \mathscr K\mid \mathcal I\subseteq \mathcal P\}$ of ${\rm Spc}(\mathscr K)$ of those primes containing $\mathcal I$.
\item The map ${\rm Spc}(q) : {\rm Spc}({\mathscr K}/ {\mathcal I})\longrightarrow {\rm Spc}(\mathscr K)$ satisfies $({\rm Spc}(q))^{-1}({\rm supp}_{\mathscr K}(x))={\rm supp}_{{\mathscr K}/ {\mathcal I}}(x)$ for each object $x$. 
\item For a subcategory $P$ of ${\mathscr K}$ with $\mathcal I \subseteq P$, we have $q(P)$ is a subcategory of ${\mathscr K}/ {\mathcal I}$ and $q^{-1}(q(P))=P$. \hfill $\square$
\end{enumerate}
\end{lem}

\section{Category algebras}

Let $k$ be a field and $\C$ be a finite category. Denote by $\Mor\C$ the finite set of all morphisms in
$\C$. The \emph{category algebra} \emph{k}$\C$ of $\C$ is defined as
follows: $\emph{k}\C=\bigoplus\limits_{\alpha \in \Mor\C}k\alpha$ as
a $k$-vector space and the product $*$ is given by the rule
 \[\alpha * \beta=\left\{\begin{array}{ll}
  \alpha\circ\beta, & \text{ if }\text{$\alpha$ and $\beta$ can be composed in $\C$}; \\
 0, & \text{otherwise.}
 \end{array}\right.\]
The unit is given by $1_{k\C}=\sum\limits_{x \in \Obj\C }\Id_x$,
where $\Id_x$ is the identity endomorphism of an object $x$ in $\C$.

Recall from \cite[Proposition 2.2]{PWebb2} that \emph{k}$\C$ is  Morita equivalent to \emph{k}$\mathscr D$ if $\C$ and $\mathscr D$ are two equivalent finite categories. In particular, $k\C$ is Morita equivalent to $k\C_0$, where $\C_0$ is any skeleton of $\C$. So we may assume that $\C$ is \emph{skeletal}, that is, for any two distinct objects $x$ and $y$ in $\C$, $x$ is not isomorphic to $y$.

Throughout the rest of this paper, we assume that $k$ is a field and $\C$ is a finite skeletal EI category if without remind.

Denote by $k$-mod the category of finite
dimensional $k$-vector spaces and $(k\text{-mod})^{\C}$ the category of covariant functors from $\C$ to $k$-mod. Recall that the category $k\C$-mod is identified with $(k\text{-mod})^{\C}$; see \cite[Proposition 2.1]{PWebb2}.

 Recall that the category $k\C$-mod is a symmetric monoidal category, write as ($k\C$-mod,$\hat{\otimes}$,$\underline{k}$). More precisely, the tensor product $\hat{\otimes}$ is defined by \[(M\hat{\otimes} N)(x)=M(x)\otimes_k N(x)\]
for any $M, N\in (k\text{-mod})^{\C}$ and $x\in {\rm Obj}\C$, and $\alpha.(m\otimes n)=\alpha.m\otimes \alpha.n$ for any $\alpha\in \Mor\C, m\in M(x), n\in N(x)$; see~\cite{XF2, XF3}.  The tensor identity $\underline{k}$ is the trivial $k\C$-module, which is also called the \emph{constant functor} sending each object to $k$ and each morphism to identity map of $k$.

Since -$\hat{\otimes}$- is exact in both variables, it gives rise to a tensor product on ${\rm D}^b(k\C)={\rm D}^b(k\C\text{-mod})$. We shall still write $\hat{\otimes}$ and $\underline{k}$ for the tensor product and tensor identity in ${\rm D}^b(k\C)$.

%The category $\C$ is called a \emph{finite EI category} provided that all endomorphisms in $\C$ are isomorphisms. In particular, $\Homc(x,x)=\Ac(x)$ is a finite group for any object $x$ in $\C$. Denote by $k\Ac(x)$ the group algebra.

There is a natural partial order on the set of objects in $\C$: $x\leq y$ if and only if $\Homc(x,y)\neq \emptyset$. This partial order in turn enables us to filtrate each $k\C$-module $M$ by group modules. Let $\C_x$ be the full subcategory of $\C$ with object $\{x\}$. Denote by $M_x=M(x)$ the subspace of $M$. It becomes a $k\C_x$-module. It also can be regarded as a $k\C$-module (but not necessarily a submodule of $M$). For each object $x$, there is a simple module $S_{x,k} : \C \rightarrow k$-mod sending $x$ to $k$ and other objects to zero. In general we have $M_x=M\hat{\otimes} S_{x,k}$; see~\cite[section 2.2]{XF2}.

\section{Spectra of derived categories}

Recall that the inclusion $\C_x\overset{\iota}{\hookrightarrow} \C$ induces a restriction
\[{\rm res}_x : k\C\text{-{\rm mod}}\longrightarrow k\C_x\text{-{\rm mod}}, \ \ \ M\mapsto M\circ \iota.\]
It is exact and preserves both tensor products and tensor identity. We write the resulting tensor derived functor as ${\rm Res}_x :  {\rm D}^b(k\C)\rightarrow {\rm D}^b(k\C_x)$; see \cite{XF2}.

Let $R$ be a left noetherian ring with a unit and $e$ be an idempotent of $R$. The Schur functor (\cite[Chapter 6]{G}) is defined to be
\[S_e=eR\otimes_R\text{-} : R\text{-{\rm mod}}\longrightarrow eRe\text{-{\rm mod}},\]
where $eR$ is viewed as a natural $eRe$-$R$-bimodule via the multiplication map. Let $\mathcal N_e$ be the full subcategory of ${\rm D}^b(R\text{-{\rm mod}})$ consisting of complex $X^{\bullet}$ with its cohomology groups ${\rm H}^n(X^{\bullet})$ lying in the kernel of $S_e$; see \cite[section 2]{XChen}. Then the Schur functor $S_e$ induces a natural equivalence of triangulated categories
\[{\rm D}^b(eRe\text{-{\rm mod}})\simeq {\rm D}^b(R\text{-{\rm mod}})/ \mathcal N_e\]
by \cite[Lemma 2.2]{XChen}, where the right hand side is a Verdier quotient category of ${\rm D}^b(R\text{-{\rm mod}})$.

Let $\C$ be a finite EI category. For each object $x$, let $e={\rm Id}_x$ be an idempotent of $k\C$. Then $k\C_x=ek\C e$. We observe that the Schur functor $S_e=ek\C\otimes_{k\C}\text{-}\simeq {\rm res}_x$ and the corresponding $\mathcal N_e:=\mathcal N^x_e=\{X^{\bullet}\in {\rm D}^b(k\C)\mid X^{\bullet}_x=0\}$. Here the $i$-th component of $X^{\bullet}_x$ is $X^i_x=X^i\hat{\otimes} S_{x,k}$, and hence we have $X_x^{\bullet}=X^{\bullet}\hat{\otimes} S_{x,k}$. Then we have the following result by \cite[Lemma 2.2]{XChen}.
\begin{rem}\label{V}
\begin{enumerate}
\item \emph{The functor 
${\rm Res}_x : {\rm D}^b(k\C)\rightarrow {\rm D}^b(k\C_x)\simeq {\rm D}^b(k\C)/ \mathcal N^x_e$ is a localization functor for each object $x$ in $\C$}. 
	
\item \emph{By Lemma~\ref{VH}, there is a homeomorphism
\[{\rm Spc}({\rm Res}_x) : {\rm Spc}{\rm D}^b(k\C_x)\overset{\sim}{\longrightarrow} V_x=\{\mathcal P\in {\rm Spc}{\rm D}^b(k\C)\mid \mathscr N^x_e\subseteq \mathcal P\},\]
where $V_x\subseteq {\rm Spc}{\rm D}^b(k\C)$ is a subspace of ${\rm Spc}{\rm D}^b(k\C)$ of those primes containing $\mathscr N^x_e$}.
\end{enumerate}	
\end{rem}

\begin{lem}\cite[Proposition 3.2.4]{XF2}\label{P}
Assume that $\mathcal P$ is a prime ideal in ${\rm Spc} {\rm D}^b(k\C)$. Then ${\rm Res}_x \mathcal P\subsetneq {\rm D}^b(k\C_x)$ for a unique $x$. Whence ${\rm Res}_x \mathcal P\in {\rm Spc} {\rm D}^b(k\C_x)$ and $\mathcal P={\rm Res}_x^{-1}({\rm Res}_x \mathcal P)$.
\end{lem}

\begin{lem}\label{EC}
Assume that $\mathcal P$ is a prime ideal in ${\rm Spc} {\rm D}^b(k\C)$. Then the following are equivalent for each object $x$ in $\C$ :
	\begin{enumerate}
		\item $S_{x,k}\notin \mathcal P$;
		\item $\mathcal N^x_e\subseteq \mathcal P$;
		\item ${\rm Res}_x \mathcal P\subsetneq {\rm D}^b(k\C_x)$.
	\end{enumerate}
\end{lem}

\begin{proof}
``(1)$\Rightarrow$ (2)"	For any $X^{\bullet}\in \mathcal N^x_e$, we have that $X^{\bullet}\hat{\otimes} S_{x,k}=X^{\bullet}_x=0\in \mathcal P$. Since $\mathcal P$ is prime and $S_{x,k}\notin \mathcal P$, we have $X^{\bullet}\in \mathcal \mathcal P$.

``(2)$\Rightarrow$ (3)" Since $\mathcal P\subsetneq {\rm D}^b(k\C)$, there is $X^{\bullet}\in {\rm D}^b(k\C)- \mathcal P$. We claim that $X^{\bullet}_x={\rm Res}_x X^{\bullet}\notin {\rm Res}_x \mathcal P$. Otherwise, since $\mathcal N^x_e\subseteq \mathcal P$, we have $X^{\bullet}\in {\rm Res}_x^{-1}({\rm Res}_x \mathcal P)=\mathcal P$ by Lemma \ref{VH} (3). This is a contradiction. Hence ${\rm Res}_x X^{\bullet}\in {\rm D}^b(k\C_x)- {\rm Res}_x \mathcal P$. Then we are done.

``(3)$\Rightarrow$ (1)" Assume $S_{x,k}\in \mathcal P$. Then we have $S_{x,k}\in {\rm Res}_x \mathcal P$. For any $X^{\bullet}\in {\rm D}^b(k\C)$, we have  ${\rm Res}_x X^{\bullet}=X^{\bullet}\hat{\otimes} S_{x,k}\in {\rm Res}_x \mathcal P$. Then we have $X^{\bullet}\in {\rm Res}_x^{-1}({\rm Res}_x \mathcal P)=\mathcal P$ by Lemma~\ref{P}. This is a contradiction.
\end{proof}

\begin{thm}\label{S}
	Let $\C$ be a finite EI category. Then there is a homeomorphism
	\[{\rm Spc} {\rm D}^b(k\C)\overset{\sim}{\longrightarrow} \bigsqcup\limits_{x\in \C} {\rm Spc} {\rm D}^b(k\C_x),\] where the right hand side is
	a disjoint union.
\end{thm}

\begin{proof}
Let $\mathcal P\in {\rm Spc} {\rm D}^b(k\C)$ be a prime ideal. There is a unique $x\in \Obj\C$ such that ${\rm Res}_x \mathcal P\neq {\rm D}^b(k\C_x)$ by Lemma~\ref{P}. Then there is a unique $x\in \Obj\C$ such that $\mathcal N^x_e\subseteq \mathcal P$ by Lemma~\ref{EC}, that is, there is a unique $x\in \Obj\C$ such that $\mathcal P\in V_x$, where $V_x=\{\mathcal P\in {\rm Spc}{\rm D}^b(k\C)\mid \mathscr N^x_e\subseteq \mathcal P\}$. Hence we have ${\rm Spc} {\rm D}^b(k\C)=\bigsqcup\limits_{x\in \C} V_x$, where the right hand side is
a disjoint union. 
There is a homeomorphism
${\rm Spc}{\rm D}^b(k\C_x)\overset{\sim}{\rightarrow} V_x$ for each object $x$ by Remark~\ref{V}. And  by Lemma~\ref{EC}, $V_x={\rm supp}_{{\rm D}^b(k\C)}(S_{x,k})$ is a close set. Then we are done. 
\end{proof}

\section{Spectra of singularity categories}

We say that $\C$ is \emph{projective over $k$} if each
$k{\rm Aut}_{\C}(y)$-$k{\rm Aut}_{\C}(x)$-bimodule $k{\rm Hom}_{\C}(x,y)$ is projective on both sides; see~\cite[Definition 4.2]{WR}. For example, a finite transporter category is a finite projective EI category; see ~\cite[Example 5.2]{WR}. We recall the fact that the category algebra $k\C$ is Gorenstein if and only if  $\C$ is projective over $k$, see ~\cite[Proposition 5.1]{WR}. If $\C$ is projective, then we have a tensor triangle equivalence
$k\C\text{-{\underline{\rm Gproj}}} \overset{\sim}{\longrightarrow} {\rm D}_{sg}(k\C)$; see~\cite{WR1,WR2}. Recall that the singularity category of $k\C$ is the Verdier quotient category ${\rm D}_{\rm sg}(k\C)={\rm D}^b(k\C)/{\rm D}^b(k\C\text{-{\rm proj}})$.

\begin{lem}\label{P1}
Assume that $\C$ is projective and $\mathcal P\in {\rm Spc} {\rm D}^b(k\C)$. Then  the following are equivalent :
\begin{enumerate}
	\item ${\rm D}^b(k\C\text{-{\rm proj}})\subseteq \mathcal P$;
	\item There is a unique object $x$ such that $\mathcal N^x_e\subseteq \mathcal P$ and ${\rm D}^b(k\C_x\text{-{\rm proj}})\subseteq {\rm Res}_x \mathcal P$.
\end{enumerate}
\end{lem}

\begin{proof}
``(1)$\Rightarrow$ (2)" Assume ${\rm D}^b(k\C\text{-{\rm proj}})\subseteq \mathcal P$. Then there is a unique object $x$ such that $\mathcal N^x_e\subseteq \mathcal P$ by Lemma \ref{P} and Lemma \ref{EC}. Let $M$ be a $k\C_x$-module. Denote by ${\rm Inc}_x M$ the functor from $\C$ to $k$-mod sending $x$ to $M(x)$ and other objects to zero. Let $X^{\bullet}\in {\rm D}^b(k\C_x\text{-{\rm proj}})$. Denote by ${\rm Inc}_x X^{\bullet}$ the complex in ${\rm D}^b(k\C)$ with the $i$-th component $({\rm Inc}_x X^{\bullet})^i={\rm Inc}_x X^i$. We claim that ${\rm Inc}_x X^{\bullet}\in {\rm D}^b(k\C\text{-{\rm proj}})$. Indeed, let $M$ be a $k\C_x$-module with finite projective dimension. Since $\C$ is projective, we have that the $k\C$-module ${\rm Inc}_x M$ has finite projective dimension by \cite[Corollary 3.6]{WR}. This implies ${\rm Inc}_x X^{\bullet}\in {\rm D}^b(k\C\text{-{\rm proj}})$. We observe that $X^{\bullet}={\rm Res}_x {\rm Inc}_x X^{\bullet}$. Since ${\rm Inc}_x X^{\bullet}\in {\rm D}^b(k\C\text{-{\rm proj}})\subseteq \mathcal P$, we have $X^{\bullet}\in {\rm Res}_x \mathcal P$. 
	
``(2)$\Rightarrow$ (1)" Assume that there is a unique object $x$ such that $\mathcal N^x_e\subseteq \mathcal P$ and ${\rm D}^b(k\C_x\text{-{\rm proj}})\subseteq {\rm Res}_x \mathcal P$. Then we have ${\rm Res}_x \mathcal P\subsetneq {\rm D}^b(k\C_x)$ by Lemma~\ref{EC}. Let $X^{\bullet}\in {\rm D}^b(k\C\text{-{\rm proj}})$. We claim that $X^{\bullet}_x={\rm Res}_x X^{\bullet}\in {\rm D}^b(k\C_x\text{-{\rm proj}})$. Indeed, let $M$ be a $k\C$-module with finite projective dimension. Since $\C$ is projective, we have that $M_x$ is a projective $k\C_x$-module by \cite[Corollary 3.6]{WR}. This implies $X^{\bullet}_x={\rm Res}_x X^{\bullet}\in {\rm D}^b(k\C_x\text{-{\rm proj}})\subseteq {\rm Res}_x \mathcal P$. Hence $X^{\bullet}\in {\rm Res}_x^{-1}({\rm Res}_x \mathcal P)=\mathcal P$ by Lemma~\ref{P}. Then we are done.
\end{proof}

%The following lemma is well known.
%\begin{lem}\label{P2}
%	Let $f : X\rightarrow Y$ be a homeomorphism of two topological spaces $X$ and $Y$. If $V$ is a close subset of $Y$, then $f^{-1}(V)\overset{\sim}{\longrightarrow} V$ is a homeomorphism.
%\end{lem}

\begin{thm}\label{PS}
	Let $\C$ be a finite projective EI category. Then there is a homeomorphism
	\[{\rm Spc} {\rm D}_{\rm sg}(k\C)\overset{\sim}{\longrightarrow} \bigsqcup\limits_{x\in \C} {\rm Spc} (kG_x\text{-{\underline{\rm mod}}})
	,\] where the right hand side is
	a disjoint union, and $G_x=\Ac(x)$. 
\end{thm}

\begin{proof}
We have $kG_x\text{-{\underline{\rm mod}}}=kG_x\text{-{\underline{\rm Gproj}}}\simeq {\rm D}_{\rm sg}(k\C_x)$ for each object $x$. Then we only need to prove that there is a homeomorphism
 \[{\rm Spc} {\rm D}_{\rm sg}(k\C)\overset{\sim}{\longrightarrow} \bigsqcup\limits_{x\in \C} {\rm Spc} {\rm D}_{\rm sg}(k\C_x).\]
 
Consider the localization functor
\[{\rm Res}_x : {\rm D}^b(k\C)\longrightarrow {\rm D}^b(k\C_x)={\rm D}^b(k\C)/\mathcal N^x_e.\]
By Lemma \ref{VH} (1), the functor ${\rm Res}_x$ induces a homeomorphism
\begin{align}\label{TP2}
{\rm Spc} {\rm D}^b(k\C_x)\overset{\sim}{\longrightarrow} V_x=\{\mathcal P\in {\rm Spc}{\rm D}^b(k\C)\mid \mathcal N^x_e\subseteq \mathcal P\},
\end{align}
where $V_x={\rm supp}_{{\rm D}^b(k\C)}(S_{x,k})$ is a close set.  

Consider the localization functor
\[q : {\rm D}^b(k\C)\longrightarrow {\rm D}_{\rm sg}(k\C)={\rm D}^b(k\C)/{\rm D}^b(k\C\text{-{\rm proj}}).\]
By Lemma \ref{VH} (1), the functor $q$ induces a homeomorphism
\begin{align}\label{TP}
{\rm Spc} {\rm D}_{\rm sg}(k\C)\overset{\sim}{\longrightarrow} V=\{\mathcal P\in {\rm Spc}{\rm D}^b(k\C)\mid {\rm D}^b(k\C\text{-{\rm proj}})\subseteq \mathcal P\},
\end{align}
where $V\subseteq {\rm Spc}{\rm D}^b(k\C)$ is a subspace of ${\rm Spc}{\rm D}^b(k\C)$ of those primes containing ${\rm D}^b(k\C\text{-{\rm proj}})$. By Lemma~\ref{P1}, we have \[V=\bigsqcup\limits_{x\in \C}\{\mathcal P\in {\rm Spc}{\rm D}^b(k\C)\mid \mathcal N^x_e\subseteq \mathcal P;{\rm D}^b(k\C_x\text{-{\rm proj}})\subseteq {\rm Res}_x \mathcal P\}:=\bigsqcup\limits_{x\in \C}V'_x,\]
where $V'_x=\{\mathcal P\in {\rm Spc}{\rm D}^b(k\C)\mid \mathcal N^x_e\subseteq \mathcal P;{\rm D}^b(k\C_x\text{-{\rm proj}})\subseteq {\rm Res}_x \mathcal P\}=\{\mathcal P\in V_x\mid {\rm D}^b(k\C_x\text{-{\rm proj}})\subseteq {\rm Res}_x \mathcal P\}$.

By Lemma \ref{VH} (2) and (3), ${\rm supp}_{{\rm D}_{\rm sg}(k\C)}(S_{x,k})=({\rm Spc}(q))^{-1}({\rm supp}_{{\rm D}^b(k\C)}(S_{x,k}))=V'_x$ for each object $x$.

Consider the localization functor
\[q' : {\rm D}^b(k\C_x)\simeq {\rm D}^b(k\C)/ \mathcal N^x_e\longrightarrow {\rm D}_{\rm sg}(k\C_x)\simeq {\rm D}^b(k\C)/\left \langle\mathcal N^x_e,{\rm D}^b(k\C_x\text{-{\rm proj}})\right \rangle,\]
where $\left \langle\mathcal N^x_e,{\rm D}^b(k\C_x\text{-{\rm proj}})\right \rangle$ denote the tensor ideal of ${\rm D}^b(k\C)$ generated by $\mathcal N^x_e$ and ${\rm D}^b(k\C_x\text{-{\rm proj}})$.

By Lemma \ref{VH} (1), the functor $q'$ induces a homeomorphism
\[{\rm Spc} {\rm D}_{\rm sg}(k\C_x)\overset{\sim}{\longrightarrow} V'_x=\{\mathcal P\in {\rm Spc}{\rm D}^b(k\C)\mid \mathcal N^x_e\subseteq \mathcal P;{\rm D}^b(k\C_x\text{-{\rm proj}})\subseteq {\rm Res}_x \mathcal P\}.\]
 
Since $V'_x$ is a close set, then we have a homeomorphism \begin{align}\label{TP1}
\bigsqcup\limits_{x\in \C} {\rm Spc} {\rm D}_{\rm sg}(k\C_x)\overset{\sim}{\longrightarrow} \bigsqcup\limits_{x\in \C} V'_x=V.
\end{align}
Then we are done by the homeomorphisms (\ref{TP}) and (\ref{TP1}).
\end{proof}

\section*{Acknowledgements}
The author is grateful to her supervisor Professor Xiao-Wu Chen for
his encouragements and discussions. This work is supported by the National Natural
Science Foundation of China (No.s 11522113, 11571329, 11671174 and 11671245), and the Fundamental Research Funds for the Central Universities.

%n--------------------------------
%n--------------------------------


\begin{thebibliography}{10}

%\bibitem{ARS}
%M. Auslander, I. Reiten, S. O. Smal{\o}, \emph{Representation theory
%of Artin algebras,} Cambridge Studies in Advanced Mathematics,
%\textbf{36}, Cambridge University %Press, Cambridge, 1995.

%\bibitem{AR}
%M. Auslander, I. Reiten, \emph{Cohen-Macaulay and Gorenstein Artin
%algebras,} in: Representation theory of finite groups and
%finite-dimensional algebras (Bielefeld, 1991), 221--245, Progr.
%Math., 95, Birkh\"{a}user, Basel, 1991.

\bibitem{Ba}
P. Balmer, \emph{The spectrum of prime ideals in tensor triangulated categories,} J. Reine Angew. Math. \textbf{588} (2005), 149--168.

\bibitem{Buch}
R.-O. Buchweitz, \emph{Maximal Cohen-Macaulay modules and Tate-cohomology over Gorenstein rings,} Unpublished Manuscript, Available at: http://hdl.handle.net/1807/16682, 1987.

\bibitem{XChen}
X.-W. Chen, \emph{Singularity categories, Schur functors and
triangular matrix rings,} Algebr. Represent. Theory \textbf{12}
(2009), no. 2-5, 181--191.

%\bibitem{DE}
%R. M. Fossum, P. A. Griffith, I. Reiten, \emph{Trivial extensions of
%abelian categories,} Lecture Notes in Mathematics, Vol.
%\textbf{456}. Springer-Verlag, Berlin-New York, 1975.
%
%\bibitem{EJ}
%E. E. Enochs, O. M. G. Jenda,\emph{Relative homological algebra,} De Gruyter Exp. Math. \textbf{30} Walter De Gruyte Co., 2000.


%\bibitem{EC}
%E. E. Enochs, M. Cort\'{e}s-Izurdiaga, B. Torrecillas,
%\emph{Gorenstein conditions over triangular matrix rings,} J. Pure
%Appl. Algebra \textbf{218} (2014), no. 8, 1544--1554.
%
\bibitem{G}
J. A. Green, \emph{Polynomial representations of $GL_n$,}  Lecture Notes in Mathematics, \textbf{830}. Springer-Verlag, Berlin-New York, 1980.

\bibitem{Kl}
S. Klein, {\em Chow groups of tensor triangulated categories}, J. Pure Appl. Algebra \textbf{220} (2016), no. 4, 1343--1381.

%\bibitem{YIw}
%Y. Iwanaga, \emph{On rings with finite self-injective dimension,}
% Comm. Algebra \textbf{7} (1979), no. 4, 393--414.

%\bibitem{LLi}
%L. P. Li, \emph{A characterization of finite EI categories with hereditary category algebras,} J. Algebra \textbf{345} (2011), 213--241.

%\bibitem{LLi1}
%L. P. Li, \emph{A generalized Koszul theory and its application,} Trans. Amer. Math. Soc. \textbf{366} (2014), 931--977.

 \bibitem{St}
G. Stevenson, {\em Support theory via actions of tensor triangulated categories}, J. Reine Angew. Math. \textbf{681} (2013), 219--254.


\bibitem{WR}
R. Wang, \emph{Gorenstein triangular matrix rings and category algebras,}
 J. Pure Appl. Algebra \textbf{220} (2016), no. 2, 666--682. 

\bibitem{WR1}
R. Wang, \emph{The MCM-approximation of the trivial module over a category algebra,}
J. Algebra Appl. Vol. \textbf{16}, No.6 (2017) 1750109 (16 pages).

\bibitem{WR2}
R. Wang, \emph{The tensor product of Gorenstein-projective modules over category algebras,} Comm. Algebra (2018), DOI: 10.1080/00927872.2018.1424864.
%\bibitem{PWebb1}
%P. Webb, \emph{Standard stratifications of EI categories and
%Alperin's weight conjecture,} J. Algebra \textbf{320} (2008), no.
%12, 4073--4091.

\bibitem{PWebb2}
P. Webb, \emph{An introduction to the representations and cohomology
of categories,} Group representation theory, 149--173, EPFL Press,
Lausanne, 2007.

%\bibitem{XF1}
%F. Xu, \emph{On local categories of finite groups,} Math. Z.
%\textbf{272} (2012), no. 3-4, 1023--1036.

\bibitem{XF2}
F. Xu, \emph{Spectra of tensor triangulated categories over category algebras,} Arch. Math. (Basel) 103 (2014), no. \textbf{3}, 235--253.

\bibitem{XF3}
F. Xu, \emph{Tensor structure on $k\C$-mod and cohomology,} Proc. Edinb. Math. Soc. (2) \textbf{56} (2013), no. 1, 349--370. 

%\bibitem{Z}
%P. Zhang, \emph{Gorenstein-projective modules and symmetric recollements,} J. Algebra \textbf{388} (2013), 65--80.


%\bibitem{Zaks}
%A. Zaks, \emph{Injective dimensions of semi-primary rings,} J.
%Algebra \textbf{13} (1969), 73--86.

\end{thebibliography}
\end{document}